\documentclass[12pt]{amsart}

\usepackage{amsfonts}
\usepackage{amssymb}
\usepackage{ifthen}
\usepackage{amscd}
\usepackage{amsxtra}
\usepackage{graphicx}
\usepackage{color}

\numberwithin{equation}{section}

\newtheorem{theorem}{Theorem}
\newtheorem{lemma}[theorem]{Lemma}
\newtheorem{corollary}[theorem]{Corollary}

\theoremstyle{definition}

\newtheorem{problem}{Problem}

\def\be{\begin{equation}}
\def\ee{\end{equation}}
\newcommand{\A}{{\mathcal A}}

\newcommand{\es}{{\mathcal S}}

\newcommand{\D}{{\mathbb D}}
\newcommand{\G}{{\mathcal G}}

\newcommand{\N}{{\mathbb N}}

\newcommand{\CC}{{\mathcal C}}
\newcommand{\F}{{\mathcal F}}

\newcommand{\ds}{\displaystyle}

\newcounter {own}
\def\theown {\thesection     .\arabic{own}}

{\qed\bigskip}

\newcounter{alphabet}
\newcounter{tmp}
\newenvironment{Thm}[1][]{\refstepcounter{alphabet}%
\bigskip%
\noindent%
{\bf Theorem \Alph{alphabet}}%
\ifthenelse{\equal{#1}{}}{}{ (#1)}%
{\bf .} \itshape}{\vskip 8pt}

\makeatletter
\newcommand{\Ref}[1]{\@ifundefined{r@#1}{}{\setcounter{tmp}{\ref{#1}}\Alph{tmp}}}
\makeatother

\newcounter{minutes}
\divide\time by 60
\newcounter{hours}
\multiply\time by 60 \addtocounter{minutes}{-\time}
\begin{document}

\bibliographystyle{amsplain}

%

\title[Inverse logarithmic coefficients problems]{Logarithmic coefficients of the inverse of univalent functions}

\thanks{
File:~\jobname .tex,
          printed: \number\day-\number\month-\number\year,
          \thehours.\ifnum\theminutes<10{0}\fi\theminutes}

\author[S. Ponnusamy]{Saminathan Ponnusamy
}
\address{
S. Ponnusamy, Department of Mathematics,
Indian Institute of Technology Madras, Chennai-600 036, India.
}
\email{samy@iitm.ac.in}

\author[N. L. Sharma]{Navneet Lal Sharma}
\address{N.L. Sharma,
S. Ponnusamy, Department of Mathematics,
Indian Institute of Technology Madras, Chennai-600 036, India.
}

\email{sharma.navneet23@gmail.com}

\author[K.-J. Wirths]{Karl-Joachim Wirths}
\address{K.-J. Wirths, Institut f\"ur Analysis und Algebra, TU Braunschweig,
38106 Braunschweig, Germany.}
\email{kjwirths@tu-bs.de}

\subjclass[2010]{30C45}
\keywords{Univalent function, Inverse function, starlike, spirallike, close-to-convex, and convex functions,
subordination, Inverse Logarithmic coefficients, Schwarz' lemma
}

\begin{abstract}
Let $\es$ be the class of analytic and univalent functions in the unit disk $|z|<1$, that have a  series of the form $f(z)=z+ \sum_{n=2}^{\infty}a_nz^n$.
Let $F$ be the inverse of the function $f\in\es$ with the series expansion 
$F(w)=f^{-1}(w)=w+ \sum_{n=2}^{\infty}A_nw^n$ for $|w|<1/4$.
The logarithmic inverse coefficients $\Gamma_n$ of $F$ are defined by the formula
$\log\left(F(w)/w\right)\,=\,2\sum_{n=1}^{\infty}\Gamma_n(F)w^n$.
In this paper, we first determine the sharp bound for the absolute value of $\Gamma_n(F)$ when $f$ belongs to
$\es$ and for all $n \geq 1$. This result motivates us to carry forward similar problems for some of its important geometric
subclasses. In some cases, we have managed to solve this question completely but in some other cases it is difficult to handle for $n\geq 4$.
For example, in the case of convex functions $f$, we show that
the logarithmic inverse coefficients $\Gamma_n(F)$ of $F$ satisfy the inequality
\[ |\Gamma_n(F)|\,\le \, \frac{1}{2n} \mbox{ for } n\geq 1,2,3
\]
and the estimates are sharp for the function $l(z)=z/(1-z)$. Although this cannot be true for $n\ge 10$, it is not clear whether
this inequality could still be true for $4\leq n\leq 9$.

\end{abstract}


\maketitle
\pagestyle{myheadings}
\markboth{S. Ponnusamy, N. L. Sharma, and K.-J. Wirths}{Inverse Logarithmic coefficients problems}

\section{Introduction}


Let $\A$ be the class of functions $f$  analytic in the unit disk $\D=\{z:\,|z|<1\}$ of the form
\begin{equation}\label{eq1}
f(z)=z+ \sum_{n=2}^{\infty}a_nz^n.
\end{equation}
The subclass of $\A$ consisting of all univalent functions $f$ in $\D$ is denoted by $\es$.
The theory of univalent functions with a strong foundation from the class $\es$ is beautiful
when it is being considered both by geometric and analytic
considerations, and in addition, logarithmic restrictions and special exponentiation methods are often useful.
During 1960's,  Milin \cite{Milin} intensively investigated the impact of transferring the
properties of the logarithmic coefficients to that of the Taylor coefficients of univalent functions themselves
or to its powers and thus, their role in the theory of univalent functions. The inequalities conjectured by
Milin  attracted much attention because
their truth would imply the truth of the Robertson conjecture and the Bieberbach
conjecture, in addition to few others \cite{Dur83,Milin,Pom75}. It is then, in 1984, Louis de Branges \cite{DeB1} proved these inequalities
and his proof resolved the most popular problem  for the class $\mathcal S$, namely, the statement
$\max_{f\in \mathcal S}|a_n|=n$ which occurs if and only if $f$ is a rotation of the Koebe function $k(z)=z/(1-z)^2$.
The proof which settles the Bieberbach conjecture  relied not on the coefficients $\{a_n\}$ of $f$ but rather the logarithmic coefficients $\{\gamma _n\}$ of $f$. Here the logarithmic coefficients $\gamma_n$ of $f\in {\es}$ are defined by the formula
\[ \log\left (\frac{f(z)}{z} \right )=2\sum_{n=1}^\infty \gamma_n(f)z^n \quad \mbox{ for } z\in \D.
\]
We use $\gamma_n(f)=\gamma_n$ when there is no confusion, and remark that some authors use $\gamma_n$ in place of $2\gamma_n$.

Let $F$ be the inverse function of $f\in\es$ defined in a neighborhood of the origin with the Taylor series expansion
\begin{equation}\label{eq3}
F(w):=f^{-1}(w)=w+ \sum_{n=2}^{\infty}A_nw^n,
\end{equation}
where we may choose  $|w|\,<\,1/4$, as we know from Koebe's $1/4$-theorem. Using a variational method, L\"{o}wner \cite{Low23} obtained the sharp estimate:
\begin{equation}\label{eq-SSWeq2}
|A_n|\leq K_n ~\mbox{  for each $n$},
\end{equation}
where $K_n = (2n)!/(n!(n + 1)!)$ and $K(w) = w + K_2w^2 + K_3w^3 +\cdots$
is the inverse of the Koebe function.  There has been a good deal of interest in determining the
behavior of the inverse coefficients of $f$ given in \eqref{eq3} when the corresponding function
$f$ is restricted to some proper geometric subclasses of $\es.$ Alternate proofs of the inequality \eqref{eq-SSWeq2}
have been given by several authors but a simpler proof was given by Yang \cite{Yang}.
As with $f$, the logarithmic coefficients $\Gamma_n, n \in {\N},$ of $F$ are defined by the equation
\begin{equation}\label{eq4}
\log\,\left(\frac{F(w)}{w}\right)\,=\,2\sum_{n=1}^{\infty}\Gamma_n(F)w^n \quad \mbox{ for } |w|<1/4.
\end{equation}
We have a natural and fundamental question.

\begin{problem}\label{prob1}
{\em
Suppose that $f\in {\es}$ or of its subclasses and $F$ is the corresponding inverse of $f$ with the form \eqref{eq3}.
If $\Gamma_n(F)$ denotes the logarithmic inverse coefficients of $F$, is it possible to determine the sharp bound for the absolute value of $\Gamma_n(F)$?
}
\end{problem}

The main aim of this article is to deal with this problem for ${\es}$ and some of its important geometric subclasses.
The article is organized as follows. In Section \ref{SSW2-sec2}, we solve this problem completely for the family $\es$
which motivates the rest of the investigation. In Section \ref{SSW2-sec3}, we introduce the classes for which we study
this problem, and present solutions to this problem in several subsections with necessary background materials.

\section{Logarithmic inverse coefficients for the class $\es$}\label{SSW2-sec2}
Before we continue to study Problem \ref{prob1} in detail, it is appropriate to deal with the class $\es$ which motivates
us to consider further investigation. Let $\es^*$ denote the class of starlike functions $f$ (i.e $f(\D)$ is a domain
starlike with respect to the origin) in $\es$.

Recall that, for $f\in \es$ and $\lambda >0$, the function $(z/f(z))^{\lambda}$ is analytic in $\D$
and has the power series expansion of the form
\begin{equation}\label{eq5}
 g(z)= \left(\frac{z}{f(z)} \right)^{\lambda}=1+ \sum_{n=1}^{\infty}b_n(\lambda, f)z^n.
\end{equation}
Throughout we use this representation.
For the logarithmic inverse coefficients $\Gamma_n$ of $F$ as given by \eqref{eq4}, the following theorem, whose proof is elegant, is fundamental in this line of discussion.

\begin{theorem}\label{thm1}
Let $f\in \es$ (or $\es^*$) and $F$ be the inverse function of $f$ and have the form $(\ref{eq3})$.
Then for $n \in{\N}$, the logarithmic inverse coefficients $\Gamma_n$ of $F$ satisfy the
sharp inequality
\[|\Gamma_n(F)|\,\leq \frac{1}{2n}\left(\begin{array}{cc}2n\\n\end{array}\right).
\]
Equality  is attained if and only if $f$  is the Koebe function or one of its rotations.
\end{theorem}
\begin{proof}
The idea of proof of here is well-known and  Lagrange's series have a similar idea of the proof. We consider
\[  \frac{d}{dw}\left(w\log\left(\frac{F(w)}{w}\right)\right)=\,\frac{wF'(w)}{F(w)}\,-\,1\,
=\,\,2\sum_{n=1}^{\infty}n\Gamma_n(F)w^n.
\]
Using the Cauchy integral formula and the relation (\ref{eq5}), 
it is easy to obtain the following identity for each $n\in {\N}$,
\begin{align}\label{eq6}
\nonumber 2n\Gamma_n(F)\, &=\frac{1}{2\pi i} \int_{C}\frac{F'(w)}{F(w) w^{n}} \, dw\\
\nonumber & =\frac{1}{2\pi i} \int_{F(C)}\left(\frac{z}{f(z)}\right)^{n}\frac{1}{z^{n+1}}\, dz \\
& =\,b_{n}(n,f),
\end{align}
where $C$ is a  Jordan curve surrounding the origin counterclockwise in $f(\D)$.
Concerning this identity, see~\cite[Theorem~3]{RW01}. With the use of L\"owner's method~\cite{Low23}, it has been proved in \cite{RW01} that
\[ |b_n(\lambda,f)| \le |b_n(\lambda,k)|=\left(\begin{array}{cc}2\lambda\\n\end{array}\right)
\quad \mbox{ for } \lambda >0,
\]
where $k$ equals the Koebe function $k(z)\,=\,z/(1-z)^2.$
Hence,
\[2n\left|\Gamma_n(F)\right|\,\leq\, b_{n}(n,k), \quad n\in {\N}.\]
Likewise, it was proved in \cite{RW01} that equality occurs if and only if $f$ equals $k$ or one of its rotations.
Since
\[ \left(\frac{z}{k(z)}\right)^n\,=\,(1-z)^{2n}=\sum_{j=0}^{n}(-1)^j\left(\begin{array}{cc}2n\\j\end{array}\right)z^j,\]
the binomial theorem implies our assertion.
\end{proof}

\section{Logarithmic inverse coefficients for the preliminary classes}\label{SSW2-sec3}

\subsection{Basic preliminary classes of study}
Let $\mathcal B$ denote the class  of all analytic functions $\phi$ in $\D$
which satisfy the condition  $|\phi(z)|<1$ for $z\in\D$. Functions in
${\mathcal B}_0:=\{\phi \in{\mathcal B} :\, \phi(0)=0\}$ are called Schwarz functions.
Let $f$ and $g$ be two analytic functions in $\D$. We say that $f$ is {\em subordinate} to $g$,
written as $f \prec g$, if there exists a function $\phi\in {\mathcal B}_0$ such that $f(z)=g(\phi(z))$
for $ z\in \D.$
In particular, if $g$ is univalent in $\D$, then $f \prec g$ is equivalent to
$f(\D)\subset g(\D)$ and $f(0)=g(0)$. 

The following subclasses of $\es $ have been studied extensively in the literature. See \cite{Go}
and \cite{OPW16,PSW18-pre,PW18} and the references therein.

\begin{enumerate}
 \item {The class $\es^*(A,B)$} is defined by
$$\es ^*(A,B):=\left \{f\in \A: \, \frac{zf'(z)}{f(z)}\prec \frac{1+Az}{1+Bz} \,
 \mbox{ for } z\in \D \right\},
$$
where $A\in\mathbb{C}$ and $-1\le B\le 0$, and this class  has been studied extensively in the literature.
For $0\le \beta <1$, $\es^*(\beta):=\es^*(1-2\beta, -1)$ is the class of starlike functions of order $\beta$.
In particular, for $B=-1$ and $A=e^{i\alpha}(e^{i\alpha}-2\beta\cos \alpha)$,
the class $\es ^*(A,B)$ reduces to the class $\mathcal{S}_\alpha (\beta )$ of spiral-like functions of order $\beta$
defined by
$$ \mathcal{S}_\alpha (\beta )
:= \left \{f\in \A: \, {\rm Re}  \left ( e^{-i\alpha}\frac{zf'(z)}{f(z)}\right )>\beta \cos \alpha, \,
z\in \D\right\},
$$
where $ \beta \in [0,1)$ and $\alpha \in (-\pi/2, \pi/2)$.
Each function in $\mathcal{S}_\alpha (\beta )$ is univalent in $\D$ (see \cite{Lib67}).
Clearly, $\mathcal{S}_\alpha (\beta )\subset \mathcal{S}_\alpha (0)\subset \mathcal{S}$
whenever $0\leq \beta <1$. Functions in $\mathcal{S}_\alpha(0)$  are called \textit{$\alpha$-spirallike}, but they do not necessarily
belong to the starlike family $\mathcal{S}^* := \mathcal{S}^*(1,-1)$.
See \cite{Dur83,Go}.

\item {The class $\G(c)$ is defined by}
\begin{align*}
\G(c)& :=\left \{f\in \A: \, {\rm Re}  \left ( 1+\frac{zf''(z)}{f'(z)}\right )
                   <1+\frac{c}{2}, \, z\in \D \right\},
\end{align*}
where $c\in (0,1]$.
Set $\G(1)=:\G$. It is known that $\G\subset \es^{*}$ and thus, functions in $\G(c)$ are starlike.
This class has been studied extensively in the recent past, see for instance \cite{OPW13,OPW18}
and the references therein.

\item {The class $\mathcal{U}(\lambda)$ is defined by}
\[ \mathcal{U}(\lambda):=\left\{f\in\A: \left|f'(z)\left(\frac{z}{f(z)}\right)^{2}-1 \right|<\lambda,
\, z\in \D\right\},
\]
where $0<\lambda\le 1$. Set $\mathcal{U}:= \mathcal{U}(1)$, and observe that $\mathcal{U}\subsetneq \es$. See \cite{Aks58, AA70}.
Many properties of $\mathcal{U}(\lambda)$ and its various generalizations have been investigated in the literature,
we refer for example \cite{OPW16,PW18} and the references therein.


\item {The class $\F(\alpha)$ is defined by}
\[ \F(\alpha):=\left\{f\in \A:\,{\rm Re} \left ( 1+\frac{zf''(z)}{f'(z)}\right )>\alpha,\, z\in\D \right\}
\]
for $\alpha\in[-1/2,1).$
In particular, we let $\F(0)=:\CC$. Functions in $\CC$ known to be convex and univalent in $\D$ (i.e $f(\D)$ is a convex domain).
For $\alpha\in[0,1)$, functions in $\F(\alpha)$ are convex functions of order $\alpha$ in $\D$,
and it is usually denoted by $\CC(\alpha)$.
The functions in $\F(-1/2)$ (and hence in $\F(\alpha)$ for $\alpha\in[-1/2,0))$ are known to be convex in one direction (and hence close-to-convex) but are not necessarily starlike in $\D$.

\end{enumerate}


\subsection{Logarithmic inverse coefficients for $f\in \es^*(A,B)$}

Throughout in the sequel, let $\mathbb{I}_k(n)$ denote the semi-closed intervals
 $\left[\frac{k}{n},\frac{k+1}{n}\right)$ for $k=0,1,\ldots, n-1$ and $ n\in\N.$

\begin{theorem}\label{thm2}
Let $f\in \es^*(A,B)$, $\delta=(1-A)/(1-B)$ with $-1\le B<A\le 1$, and $k_{A,B;n}(z)= z(1+Bz^n)^{(A-B)/{nB}}$. Then
for $n \in{\N}$, the logarithmic inverse coefficients $\Gamma_n$ of $F$ satisfy the following inequalities:
\begin{enumerate}
\item when $n\in \N$ and $n(1-\delta)\notin \N$, we have
\begin{equation}\label{thm1-eq1}
|\Gamma_n(F)|\,\le \frac{1}{2n}\, \prod_{j=0}^{n-1}\, \frac{n(A-B)+Bj}{1+j} \quad
\mbox{ for } \delta\in \mathbb{I}_0(n)=[0,1/n).
\end{equation}
\item when $n\in \N$ and  $\delta\in \mathbb{I}_k(n),\, k=1,2,\ldots, n-1$, we have
\begin{equation}\label{thm1-eq2}
|\Gamma_n(F)|\,\le\frac{n-k}{2n^2} \,\prod_{j=0}^{n-k-1}\, \frac{n(A-B)+Bj}{1+j} .
\end{equation}
\item when $n\in \N$ and $n(1-\delta)\in \N$, $(\ref{thm1-eq1})$ holds for
$\delta\in \mathbb{I}_1(n)$,
and $(\ref{thm1-eq2})$ holds for  $\delta\in \mathbb{I}_k(n), k=2,3,\ldots, n-2$.

\item for $\delta \in \mathbb{I}_{n-1}(n)$, we have
\begin{equation}\label{thm1-eq3}
 |\Gamma_n(F)|\, \le \frac{A-B}{2n}, \quad n\in \N.
 \end{equation}
\end{enumerate}
The inequalities  \eqref{thm1-eq1} and \eqref{thm1-eq3}  are sharp for the functions $k_{A,B;1}(z)$ and
$k_{A,B;n}(z)$,  respectively.
\end{theorem}
\begin{proof}
Suppose $f\in \es^*(A,B)$.
From the relation (\ref{eq6}), we have
\[2n\Gamma_n(F) =b_{n}(n,f),\quad n\in {\N},
\]
where $b_{n}(n,f)$ is defined by (\ref{eq5}).
In order to compute $|\Gamma_n(F)|$, we shall have to estimate $|b_{n}(n,f)|$ and for this, we use \cite[Theorem~2.7]{AV17}.
It is worth to remark that one could use the analysis used in \cite{JPR95}.
First, take $\lambda = n \in \N$, we note that the inequality (2.9) in \cite{AV17} is
applicable for $k=0$ (in case of $n(1-\delta)\in \N$, the inequality (2.9) in \cite{AV17}
is also applicable but only for $k=1$). Therefore for $\delta\in \mathbb{I}_0= [0,1/n)$, 
the inequality (2.9) in \cite{AV17} yields
\[ |\Gamma_n(F)|=\frac{1}{2n}|b_{n}(n,f)|\,\le \frac{1}{2n}\, \prod_{j=0}^{n-1}\, \frac{n(A-B)+Bj}{1+j}
\quad \mbox{ for } n \in \N,
\]
which is precisely the inequality (\ref{thm1-eq1}).
The equality holds for $k_{A,B;1}(z)=  z(1+Bz)^{(A-B)/B}$. We have
\[ \left(\frac{z}{k_{A,B;1}(z)}\right)^n = (1+Bz)^{-\xi}
= \sum_{m=0}^{\infty} \frac{(-1)^m(\xi)_m\, B^m}{(m)!}z^m,
\]
where $\xi= (A-B)n/B$ and $(a)_m=\Gamma (a+m)/\Gamma (a)$ denotes the  Pochhammer symbol.
Similarly, for $\delta\in \mathbb{I}_k(n),\, k=1,2,\ldots, n-1$, the inequality (2.10) in \cite{AV17}
gives
\[ |\Gamma_n(F)|=\frac{1}{2n}|b_{n}(n,f)|\,
\le\frac{n-k}{2n^2} \,\prod_{j=0}^{n-k-1}\, \frac{n(A-B)+Bj}{1+j}.
\]
This gives (\ref{thm1-eq2}). Finally, for $\delta\in \mathbb{I}_{n-1}(n)$, the inequality (2.11)
in \cite{AV17} yields
\[ |\Gamma_n(F)|=\frac{1}{2n}|b_{n}(n,f)|=\frac{A-B}{2n}, \quad n\in \N,
\]
which gives (\ref{thm1-eq3}). It is easily verified that equality holds in (\ref{thm1-eq3}) as
the function
$$k_{A,B;n}(z)= z(1+Bz^n)^{(A-B)/{nB}}
$$
demonstrates. This completes the proof of Theorem~\ref{thm2}.
\end{proof}

Theorem \ref{thm2} for the case $A=1-2\beta$ and $B=-1$ takes the following simple form.

\begin{corollary}\label{PSW-cor1}
Let $f\in\es^*(\beta)$ for some $\beta\in[0,1)$, and $k_{\beta;n}(z)=z/(1-z^n)^{2(1-\beta)/n}$.
Then the logarithmic inverse coefficients $\Gamma_n$ of $F$ satisfy the inequalities:
\begin{enumerate}
\item for $n\in \N$ and $\beta\in [0,1/n)$, we have
\begin{equation}\label{cor1-eq1}
|\Gamma_n(F)|\,\le \frac{1}{2n}\, \prod_{j=0}^{n-1}\, \frac{2n(1-\beta) -j}{1+j},
\end{equation}
\item 
for $n\in \N$ and  $\beta\in \mathbb{I}_k(n),\, k=1,2,\ldots, n-1$, we have
\[ |\Gamma_n(F)|\,\le\frac{n-k}{2n^2} \,\prod_{j=0}^{n-k+1}\, \frac{2n(1-\beta) -j}{1+j}
\]
\item for $\beta\in \mathbb{I}_{n-1}(n)$, we have
\begin{equation}\label{cor1-eq2}
 |\Gamma_n(F)|\, \le \frac{1-\beta}{n}, \quad n\in \N.
 \end{equation}
\end{enumerate}
The inequalities \eqref{cor1-eq1} and \eqref{cor1-eq2} are sharp for the functions $k_{\beta;1}(z)$ and
$k_{\beta;n}(z)$,   respectively.
\end{corollary}
We remark that when $A=1$ and $B=-1$ in Theorem~\ref{thm2}, or when $\beta =0$ in Corollary \ref{PSW-cor1},
we obtain Theorem~\ref{thm1} for $f\in \es^*$. Moreover, we can generalize Corollary \ref{PSW-cor1} for the class $\mathcal{S}_\alpha (\beta )$
of spiral-like functions of order $\beta$.

\subsection{Logarithmic inverse coefficients for $\mathcal{S}_\alpha (\beta )$}
In \cite{JPR95}, the authors proved the following theorem (which we state in our form).

\begin{Thm}
Suppose $f(z)=z+ \sum_{n=p+1}^{\infty}a_nz^n \in \es_{\alpha}(\beta)\,(|\alpha|<\pi/2,\,0\le\beta<1)$, and for integral
$t\ge 1$, let
\[  \left (\frac {z}{f(z)}\right )^{t}=1+ \sum_{k=p}^{\infty}b_{k}^{(p)}(t,f)\,z^k, \quad 0<|z|<1.
\]
Then
\[  |b_{k}^{(p)}(t,f)| \le \frac{m\,p}{k} \prod_{j=0}^{m-1} \left(\frac{\left|(2t/p)(1-\beta)\cos \alpha \, e^{-i\alpha}-j \right|}{j+1} \right)
\]
for
\[ -mp \le k \le (m+1)p-1,
\]
 where $m=1,\ldots,M+1$, and $M=[t(1-\beta)/p]$. Here $[x]$ denotes the largest integer not exceeding $x$.
\end{Thm}

Setting  $p=1$ gives that if $f(z)=z+ \sum_{n=2}^{\infty}a_nz^n \in \es_{\alpha}(\beta)$ and $b_{k}^{(1)}(t,f)=:b_{k}(t,f)$, then we have
\begin{equation}\label{eq6-n}
 |b_{k}(t,f)| \le  \prod_{j=0}^{k-1} \left(\frac{\left|2t(1-\beta)\cos \alpha \, e^{-i\alpha}-j \right|}{j+1} \right)
\end{equation}
 where $k=1,\ldots,M+1$, and $M=[t(1-\beta)]$. Moreover,
\[2n\Gamma_n(F) =b_{n}(n,f),\quad n\in \mathbb{N},
\]
and thus, by taking $t = n \in \mathbb{N}$ in (\ref{eq6-n}), we obtain
\[ | \Gamma_n(F)| \le \frac{1}{2n} \prod_{j=0}^{k-1} \left(\frac{\left|2n(1-\beta)\cos \alpha \, e^{-i\alpha}-j \right|}{j+1} \right)
\]
 where $k=1,\ldots,[n(1-\beta)]+1$.

%
This is the basic and we organize it in the following form.
We use results from \cite{AV17} and Theorem~\ref{thm2} to prove the following.

\begin{theorem}\label{thm3}
Let $f\in\es_{\alpha}(\beta)$ for some $ \beta \in [0,1)$ and $\alpha \in (-\pi/2, \pi/2)$.
Then the logarithmic inverse coefficients of $F$ satisfy the inequalities:
\begin{enumerate}
\item  for $n\in \N$ and  $\beta \in \mathbb{I}_0(n)=[0,1/n)$, we have
\begin{equation}\label{thm3-eq1}
|\Gamma_n(F)|\,\le \frac{1}{2n}\, \prod_{j=0}^{n-1}\, \frac{|2n(1-\beta)e^{-i\alpha}\cos{\alpha} -j|}{1+j}
\end{equation}
\item
for $n\in \N$ and  $\beta\in \mathbb{I}_k(n),\, k=1,2,\ldots, n-1$, we have
 \begin{equation}\label{thm3-eq2}
|\Gamma_n(F)|\,\le \frac{n-k}{2n^2}\, \prod_{j=0}^{n-k-1}\, \frac{|2n(1-\beta)e^{-i\alpha}\cos{\alpha} -j|}
{1+j} .
\end{equation}
 \item for $\beta \in \mathbb{I}_{n-1}(n)$, we have
\begin{equation}\label{thm3-eq0}
|\Gamma_n(F)|\, \le \frac{(1-\alpha)\cos\beta}{n} .
\end{equation}
\end{enumerate}
The estimates $(\ref{thm3-eq0})$ and $(\ref{thm3-eq1})$ are sharp for
$f_{\alpha, \beta;n}(z)=z/(1-z^n)^{\gamma/n}$,~ $\gamma=2(1-\beta)\cos \alpha$
and $f_{\alpha, \beta;1}(z)$, respectively.
\end{theorem}

\begin{proof}
Suppose $f\in \es_{\alpha}(\beta)$.
From the relation (\ref{eq6}), we have
\[2n\Gamma_n(F) =b_{n}(n,f),\quad n\in {\N},
\]
where $b_{n}(n,f)$ is defined by (\ref{eq5}).
In order to find $|\Gamma_n(F)|$, we need to estimate $|b_{n}(n,f)|$ with the help of \cite[Theorem~4]{XLS13}.
First, take $\lambda = n \in \N$, we note that the inequality (47) in \cite{XLS13} is
applicable only for $k=0$. Therefore for $\beta \in [0,1/n)$,
the inequality (47) in \cite{XLS13} yields
\[ |\Gamma_n(F)|=\frac{1}{2n}|b_{n}(n,f)|\,\le \frac{1}{2n}\,
\prod_{j=0}^{n-1}\,\frac{|2n(1-\beta)e^{-i\alpha}\cos{\alpha} -j|}{1+j}
\quad \mbox{ for } n \in \N,
\]
which is precisely the inequality (\ref{thm3-eq1}).
The equality holds for $f_{\alpha, \beta}(z)=z/(1-z)^{\gamma}$,~ $\gamma=2(1-\beta)\cos \alpha$. We note that
\[ \left(\frac{z}{f_{\alpha, \beta}(z)}\right)^n = (1-z)^{-\theta}
= \sum_{m=0}^{\infty} \frac{(\theta)_m}{(m)!}z^m.
\]
 where  $\theta=-n\gamma$.
Similarly, for $\beta\in \mathbb{I}_k(n),\, k=1,2,\ldots, n-1$,
the inequality (48) in \cite{XLS13} yields (\ref{thm3-eq2}).
Finally, for $\beta \in  \mathbb{I}_{n-1}(n)$, we note that the inequality (49) in \cite{XLS13}
gives
$$|\Gamma_n(F)|=\frac{1}{2n}|b_n(n,f)| \le \frac{(1-\alpha)\cos \beta}{n}, \quad n\in\N,
$$
which establishes (\ref{thm3-eq0}). It is easily verified that equality holds in (\ref{thm3-eq0})
for the function $f_{\alpha, \beta;n}(z)=z/(1-z^n)^{\gamma/n}$. This completes the proof of Theorem~\ref{thm3}.
\end{proof}

\subsection{Logarithmic inverse coefficients for $\mathcal{G}(c)$}
\begin{lemma}\label{lem1}
 Let $f\in \mathcal{G}(c)$ for some $c\in (0,1]$ and for each fixed $\lambda >0$, the Taylor coefficients
 $b_m(\lambda,f)$  be given by $(\ref{eq5})$. Then
\begin{enumerate}
\item for $\lambda \in(0,1]$, we have
\begin{equation}\label{lem1-eq1}
|b_m(\lambda,f)| \le \frac{\lambda c}{m(1+c)} \quad \mbox{ for } m=1,2,\ldots;
\end{equation}
\item for $\lambda >1$, we have
\begin{equation}\label{lem1-eq2}
|b_m(\lambda,f)|\le \frac{1}{(1+c)^m}\prod_{j=0}^{m-1}\,\frac{\lambda c+j}{1+j}
\quad \mbox{ for } m=1,2,\ldots,[\lambda]+1;
\end{equation}
  and
\begin{equation}\label{lem1-eq3}
 |b_m(\lambda,f)| \le \frac{[\lambda]}{m(1+c)^{[\lambda]}}\prod_{j=0}^{[\lambda]-1}\,
 \frac{\lambda c+j}{1+j}
\quad \mbox{ for } m=[\lambda]+2,[\lambda]+3,\ldots;
\end{equation}
\end{enumerate}
The estimates $(\ref{lem1-eq1})$ and $(\ref{lem1-eq2})$ are sharp for the function
$f'_{c,m}(z)=(1-z^m)^{\frac{c}{m}}$ and $f'_{c;1}(z)=f'_c(z)$, respectively.
\end{lemma}

\begin{proof}
Suppose that $f\in {\mathcal G}(c)$. Then we have (see  \cite{PR95}) 
 \[ \frac{zf'(z)}{f(z)} -1\prec  \frac{(1+c )(1-z)}{1+c -z}-1= \frac{-cz}{1+c-z}.
 \]
As
$$g(z)= \left(\frac{z}{f(z)} \right)^{\lambda}=1+ \sum_{n=1}^{\infty}b_n(\lambda, f)z^n,
$$
by the definition of subordination, there exists an analytic function
$\varphi \in \mathcal{B}_0$ 
such that
\[ \frac{zg'(z)}{-\lambda g(z)} = \frac{zf'(z)}{f(z)} - 1 = \frac{-c\varphi(z)}{1+c-\varphi(z)}
\]
or equivalently
$$(1+c)zg'(z)= \varphi(z)(\lambda cg(z) +zg'(z) ).
$$
As with the standard procedure, we may write this in series form as
\[ (1+c)\sum_{k=1}^{m}kb_k(\lambda,f)z^k + \sum_{k=m+1}^{\infty}d_k(\lambda,f)z^k =
\varphi(z)\left(\lambda c + \sum_{k=1}^{m-1}(\lambda c + k)b_k(\lambda,f)z^k \right),
\]
the second sum on the left-hand side being convergent in $\D$. By Clunie's method~\cite{Clu59,CK60}
(see also Parseval-Gutzmer formula) together with
$|\varphi(z)|< 1$ gives
\begin{equation}\label{lem1-eq4}
 (1+c)^2m^2|b_m(\lambda,f)|^2 \le
\lambda^2 c^2 + \sum_{k=1}^{m-1}\,[(\lambda c + k)^2-k^2(1+c)^2]\,|b_k(\lambda,f)|^2.
\end{equation}
Since $(\lambda c + k)^2-k^2(1+c)^2=c(\lambda -k)[c(\lambda +k)+2k]$,
the sign of each term inside the summation symbol on the right-hand side of (\ref{lem1-eq4})
depends on the expression $(\lambda -k)$ for $k=1,2,\ldots, m-1$. \\

\noindent
{\bf Case I:} If $\lambda\in(0,1]$, then $\lambda -k\le 0$ for $k=1,2,\ldots, m-1$.
Then from (\ref{lem1-eq4}), we find that
\[ |b_m(\lambda,f)| \le \frac{\lambda c}{m(1+c)} \quad \mbox{ for } m=1,2,\ldots ,
\]
which establishes the inequality (\ref{lem1-eq1}).\\

\noindent
{\bf Case II:} If $\lambda >1,$ then $\lambda -k >0$ for $k=1,2,\ldots, [\lambda]$,
and $\lambda -k \le 0$ for $k=[\lambda]+1,[\lambda]+2,\ldots$.
Therefore from (\ref{lem1-eq4}), for $m=1,2,\ldots, [\lambda]+1$, we obtain
\begin{equation}\label{lem1-eq5}
m^2|b_m(\lambda,f)|^2 \le \frac{1}{(1+c)^2}\left(
\lambda^2 c^2 + \sum_{k=1}^{m-1}\,[(\lambda c + k)^2-k^2(1+c)^2]\,|b_k(\lambda,f)|^2 \right).
\end{equation}
Now we use the principle of mathematical induction on $m$. For $m=1$, it follows from
(\ref{lem1-eq5}) that $|b_1(\lambda,f)| \le \lambda c/(1+c)$. This gives the estimate
(\ref{lem1-eq2}) for $m=1$. For $m=2,\ldots,[\lambda]$, we now assume that
\begin{equation}\label{lem1-eq6}
 |b_m(\lambda,f)| \le \frac{1}{(1+c)^m}\prod_{j=0}^{m-1}\,\frac{\lambda c+j}{1+j}
\end{equation}
holds. Then, using (\ref{lem1-eq5}), (\ref{lem1-eq6}) and simplifying,  it follows that
\begin{align*}
m^2|b_m(\lambda,f)|^2 & \le \frac{1}{(1+c)^2}\left[
\lambda^2 c^2 + \sum_{k=1}^{m-1}\,\left ((\lambda c + k)^2-k^2(1+c)^2\right )\,
  \frac{1}{(1+c)^{2k}}\prod_{j=0}^{k-1}\,\left(\frac{\lambda c+j}{1+j}\right)^2 \,\right]\\
& = \frac{1}{(1+c)^2}\left[
\lambda^2 c^2 + \sum_{k=1}^{m-1}\,(A_{k+1}^2-A_k^2) \right ], \quad
A_k=\frac{k}{(1+c)^{k-1}}\prod_{j=0}^{k-1}\, \frac{\lambda c+j}{1+j} \, , \\
& =\frac{1}{(1+c)^2} A_m^2.
\end{align*}
Hence, for $m=1,2,\ldots,[\lambda]+1$, we have
\[  |b_m(\lambda,f)|\le \frac{1}{(1+c)^m}\prod_{j=0}^{m-1}\,\frac{\lambda c+j}{1+j}.
\]
This establishes the inequality (\ref{lem1-eq2}).\\

\noindent
{\bf Case III:}
Now, we will prove the inequality (\ref{lem1-eq3}). Recall that if $\lambda >1,$
then $\lambda -k \le 0$ for $k=[\lambda]+1,[\lambda]+2,\ldots$.  From (\ref{lem1-eq4}),
for $m=[\lambda]+2,[\lambda]+3,\ldots$, we get
\begin{equation}\label{lem1-eq7}
m^2|b_m(\lambda,f)|^2 \le \frac{1}{(1+c)^2}\left[\lambda^2 c^2 +
\sum_{k=1}^{[\lambda]-1}\,[(\lambda c + k)^2-k^2(1+c)^2]\,|b_k(\lambda,f)|^2 \right].
\end{equation}
Using (\ref{lem1-eq7}) and mathematical induction hypothesis (\ref{lem1-eq6}), we get as before
\begin{align*}
m^2|b_n(\lambda,f)|^2 & \le \frac{1}{(1+c)^2}\left[
\lambda^2 c^2 + \sum_{k=1}^{[\lambda]-1}\,\left ((\lambda c + k)^2-k^2(1+c)^2\right )\,
  \frac{1}{(1+c)^{2k}}\prod_{j=0}^{k-1}\,\left(\frac{\lambda c+j}{1+j}\right)^2 \,\right]\\
& =\frac{1}{(1+c)^{2[\lambda]}(([\lambda]-1)!)^2}\prod_{j=0}^{[\lambda]-1}\,\left(\lambda c+j\right)^2.
\end{align*}
Hence,
\[  |b_m(\lambda,f)| \le \frac{[\lambda]}{m(1+c)^{[\lambda]}}
\prod_{j=0}^{[\lambda]-1}\,\frac{\lambda c+j}{1+j}
\quad \mbox{ for } m=[\lambda]+2,[\lambda]+3,\ldots.
\]
This establishes the inequality (\ref{lem1-eq3}).
\end{proof}
If we take $c=1$ in Lemma~\ref{lem1}, we get the following result.	

\begin{corollary}
 Let $f\in \mathcal{G}(1)$ and for each fixed $\lambda >0$, let the Taylor coefficients $b_m(\lambda,f)$
 be given by $(\ref{eq5})$. Then
\begin{enumerate}
\item for $\lambda \in(0,1]$, we have
\begin{equation}\label{cor2-eq0}
 |b_m(\lambda,f)|\le \frac{\lambda}{2m} \quad \mbox{ for } m=1,2,\ldots;
\end{equation}

\item for $\lambda >1$, we have
\begin{equation}\label{cor2-eq1}
|b_m(\lambda,f)|\le  \frac{1}{2^m}\,\prod_{j=0}^{m-1}\frac{\lambda + j}{1+j}
\quad \mbox{ for } m=1,2,\ldots,[\lambda]+1;
\end{equation}
  and
\[ |b_m(\lambda,f)|\le \frac{[\lambda]}{m\,2^{[\lambda]}} \, \prod_{j=0}^{[\lambda]-1} \frac{\lambda + j}{1+j}
\quad \mbox{ for } m=[\lambda]+2,\ldots.
\]
\end{enumerate}
 The estimate $(\ref{cor2-eq0})$ is sharp for $f'_{1,m}$ and the estimate $(\ref{cor2-eq1})$ is sharp for $f_1(z)=z-z^2/2$.
\end{corollary}

Now we are ready to state our next main result.

\begin{theorem}
Let $f\in \mathcal{G}(c)$ for some $c\in (0,1]$. Then the logarithmic inverse coefficients $\Gamma_n$ of $F$ satisfy the inequality
\[ |\Gamma_n(F)|\,\le \, \frac{1}{2n(1+c)^n}
\prod_{j=0}^{n-1}\, \frac{nc+j}{(1+j)} \quad \mbox{ for $n\in \N$}.
\]
The result is best possible for the function $f'_c(z)=(1-z)^c$.
\end{theorem}
\begin{proof}
Suppose $f\in {\mathcal G}(c)$.
From the relation (\ref{eq6}), we have
\[2n\Gamma_n(F) =b_{n}(n,f),\quad n\in {\N},
\]
where $b_{n}(n,f)$ is defined by (\ref{eq5}).
In order to find $|\Gamma_n(F)|$, we shall estimate $|b_{n}(n,f)|$ using Lemma~\ref{lem1}.
For $\lambda =n \in\N$, we note that the inequalities (\ref{lem1-eq1}) and (\ref{lem1-eq2}) are
applicable. Therefore, the inequalities (\ref{lem1-eq1}) and (\ref{lem1-eq2}) yield
\[ |\Gamma_n(F)|= \frac{1}{2n} |b_n(n,f)| \quad \mbox{ for } n \in\N.
\]
The desired conclusion follows.
\end{proof}

\begin{corollary}
Let $f\in \mathcal{G}(1)$. Then 
\[|\Gamma_n(F)|\,\le \,
\frac{(2n-1)!}{(n!)^2 \,2^{n+1}}
 \quad \mbox{for $n\in \N$}.
\]
The result is best possible for the function $f_0(z)=z-z^2/2$.
\end{corollary}

\subsection{Logarithmic inverse coefficients for $\mathcal{U}(\lambda)$}

Now, we will discuss the logarithmic inverse coefficients $\Gamma_n$ for the class $\mathcal{U}(\lambda)$.
It is a simple exercise to see that $f\in\mathcal{U}(\lambda)$ if and only if
\begin{equation}\label{eq1-thm6}
f(z)\,=\,\frac{z}{1\,-\,a_2z\,+\,\lambda z\int_0^z\omega(t)\,dt},
\end{equation}
where $2a_2\,=\,f''(0)$, $\omega$ is analytic and $|\omega(z)|\leq 1$ for $|z|\,<\,1$. Moreover, we also see
from \eqref{eq1-thm6} that
$$f'(z)\left(\frac{z}{f(z)}\right)^{2}-1 =-\lambda z^2\omega(z),
$$
where $\lambda \omega(0)= -(a_3-a_2^2)$. In~\cite{PW18}, the authors proved that if $\omega(0)\,=\,a\in \D$ and
\[v(x)\,=\,\int_0^1\frac{x\,+\,t}{1\,+\,xt}\,dt\,=\,\frac{1}{x}\,-\,\frac{1\,-\,x^2}{x^2}\,
\log (1\,+\,x),\,\,x\in [0,1],\]
where $v(0)=\displaystyle{\lim_{x\rightarrow 0^{+}}v(x)}=1/2$, then we have
the sharp inequality
\begin{equation}\label{eq2-thm6}
|a_2|\,\leq\, 1\,+\, \lambda v(|a|).
\end{equation}

\begin{theorem}\label{PSW-thU}
Let $f\in\mathcal{U}(\lambda)$ for $0<\lambda\leq 1$. Then the logarithmic inverse coefficients $\Gamma_n$
of $F$ satisfy the inequality
\[ |\Gamma_1(F)|\,\le \, \frac{1}{2}\left[1\,+\, \lambda v(|a|)\right]
\mbox{ and }
|\Gamma_2(F)|\,\le \frac{1}{4}\left[(1\,+\, \lambda v(|a|))^2\,+\,2\lambda |a|\right].
\]
Equality is achieved in both inequalities for the function
\be\label{eq-SSWeq1}
f(z)\,=\,\frac{z}{1\,-\,(1\,+\, \lambda v(a))z\,+\,\lambda z\int_0^z\frac{t\,+\,a}{1\,+\,at}\,dt},
\ee
where $a\in (0,1)$.
\end{theorem}
\begin{proof}
Suppose  $f\in\mathcal{U}(\lambda)$. Then from $(\ref{eq1-thm6})$, we have
\[ \frac{z}{f(z)}\,=\,1\,-\,a_2 z\,+\lambda a z^2\,+o(z^2),\mbox{ for }\,z\to 0.
\]
From (\ref{eq6}), we know that $2n\Gamma_n(F)=b_n(n,f)$ for $n\in\N$, and from the last relation it follows easily that
$$b_1(1,f)=-a_2 ~\mbox{ and }~b_2(2,f)=3a_2^2-2a_3=a_2^2\,+2\lambda a.
$$
Hence, by using $(\ref{eq2-thm6})$, we get
\[ 2|\Gamma_1(F)|\,= \,|a_2|\,\leq\,1\,+\, \lambda v(|a|),\]
and
\[ 4|\Gamma_2(F)|\,=\,|a_2^2\,+2\lambda a|\,\leq\,(1\,+\, \lambda v(|a|))^2\,+\,2\lambda |a|.
\]
Equality case is easy to obtain from \eqref{eq-SSWeq1}.
\end{proof}

%
\subsection{Logarithmic inverse coefficients for $\F(\alpha)$}

We see from the definition of $\F(\alpha)$ that if $f\in \F(\alpha)$, then
\[ 1+\frac{zf''(z)}{f'(z)}\,\prec \,\frac{1+(1-2\alpha)z}{1-z}, ~
\, \mbox{ i.e. }~ \,  \frac{zf''(z)}{f'(z)}\,\prec \,\frac{2(1-\alpha)z}{1-z}.
\]
 By the definition of subordination, we get 
\[ \frac{zf''(z)}{f'(z)}\,=\,\frac{2(1-\alpha)\varphi(z)}{1\,-\,\varphi(z)}, ~\mbox{ i.e. }~
zf''(z)(1\,-\,\varphi(z))=   2(1-\alpha)f'(z)\varphi(z),
\]
where $\varphi \in \mathcal{B}_0.$ Using the Taylor expansion
 $\varphi(z)=\sum_{k=1}^{\infty}c_kz^k$ and of $f(z)$ given by \eqref{eq1},
 we can write the above relation in the series representation
\begin{align*}
 a_2z+ (3a_3-a_2c_1)z^2 & +(6a_4-3a_3c_1-a_2c_2)z^3+\ldots  \\
 & = (1-\alpha)\left[c_1z+ (2a_2c_1+c_2)z^2
 + \left(3a_3c_1 + 2a_2c_2 + c_3 \right)z^3+\ldots \right]
\end{align*}
and the sharp inequality $|c_n|\leq 1-|c_1|^2 $ holds for $n\geq 2$. Now, we compare the coefficients
of $z^n$ for $n=2,3,4$ and get
\begin{equation}\label{eq0-thm-C}
\left \{ \begin{array}{l}
\ds a_2=(1-\alpha)c_1 \\[1mm]
\ds 3a_3= (1-\alpha)((3-2\alpha)c_1^2\,+\,c_2) \\[1mm]
\ds 6a_4=(1-\alpha)((2-\alpha)(3-2\alpha)c_1^3\,+\,(5-3\alpha)c_1c_2\,+\,c_3)
 \end{array} \right.
\end{equation}
In view of the relation \eqref{eq3}, we have
$$f(F(w))=w, \quad F(0)=0=f(0)~\mbox{ and }~ F'(0)=1=f'(0),
$$
where $z=F(w)$. Differentiating this we find that $f'(z)F'(w) =1$, and further differentiation gives
$$\left \{
\begin{array}{l}
f''(z)(F'(w))^2+f'(z)F''(w)=0,\\
f'''(z)(F'(w))^3+ 3f''(z)F'(w)F''(w)+f'(z)F'''(w)=0,\\
f^{(iv)}(z)(F'(w))^4+ 6f'''(z)(F'(w))^2F''(w)
+f''(z)[3(F''(w))^2+3F'(z)F'''(w)+F'''(w)]\\
~~~~~~~~~~+f'(z)F^{(iv)}(w)=0.
\end{array}
\right .
$$
Setting $z=0$ and $w=0$, we obtain that
\begin{equation}\label{eq2-thm-C}
\left \{ \begin{array}{l}
\ds A_2=-a_2 \\[1mm]
\ds A_3=-a_3+2a_2^2\\[1mm]
\ds A_4=-a_4+5a_2\,a_3-5a_2^3.
 \end{array} \right.
\end{equation}
Next, we simplify (\ref{eq4}) and write in the series form
\begin{align*}
 A_2w+A_3w^2+A_4w^3+\cdots  -\frac{1}{2}[A_2w+A_3w^2+\cdots ]^2 +\frac{1}{3}[A_2w+\cdots ]^3 +\cdots
 ~~~=2\sum_{n=1}^{\infty}\Gamma_n(F)w^n
\end{align*}
Now, we compare the coefficients of $w^n$ for $n=1,2,3$ and find that
\begin{equation}\label{eq1-thm-C}
\left \{ \begin{array}{l}
\ds 2\,\Gamma_1(F) =A_2 \\[1mm]
\ds 2\,\Gamma_2(F)=A_3-\frac{1}{2}\,A_2^2\\[1mm]
\ds 2\,\Gamma_3(F)=A_4-A_2\,A_3+\frac{1}{3}\,A_2^3.
 \end{array} \right.
 \end{equation}
From the formulas (\ref{eq1-thm-C}) and (\ref{eq2-thm-C}), we obtain
\begin{equation}\label{eq2.1-thm-C}
\left \{ \begin{array}{l}
\ds 2\Gamma_1(F)\,=\,-a_2 \\[1mm]
\ds 4\Gamma_2(F)|= -2a_3\,+\,3a_2^2\\[1mm]
\ds  6\Gamma_3(F)\, =\,-3a_4\,+\,12a_2a_3\,-\,10a_2^3.
 \end{array} \right.
\end{equation}
Finally, the formulas (\ref{eq0-thm-C}) and (\ref{eq2.1-thm-C}) together yield
\begin{equation}\label{basic}
\left \{ \begin{array}{l}
\ds 2\Gamma_1(F)\,=\,-(1-\alpha)c_1 \\[1mm]
\ds 4\Gamma_2(F)|= \frac{1-\alpha}{3}(-2c_2\,+\,(3\,-5\alpha)c_1^2)\\[1mm]
\ds  6\Gamma_3(F)\, =\,\frac{1-\alpha}{2}\, (- c_3 \,+\,(3-5\alpha)\,c_1\,c_2\,-\, (3\alpha -2)(2\alpha -1)\,c_1^3).
 \end{array} \right.
\end{equation}
These equations imply the following sharp bounds for the logarithmic inverse coefficients.

The first equation in \eqref{basic} gives

\begin{theorem}\label{thm1-thm-C}
Let $f\in\F(\alpha)$ for some  $\alpha\in[-1/2,1).$ Then
\[|\Gamma_1(F)|\,\leq\,\frac{1\,-\,\alpha}{2}.\]
Equality is attained if and only if $f'(z)\,=\,(1-z)^{-2(1-\alpha)}$ or a rotation of this function.
\end{theorem}

The second and the third relations in \eqref{basic} give

\begin{theorem}\label{thm2-thm-C}
Let $f\in\F(\alpha)$ for some  $\alpha\in[-1/2,1).$ Then
\begin{enumerate}
\item[{\rm (a)}] If  $\alpha\in [-1/2,1/5]$, then
\[|\Gamma_2(F)|\,\leq\,\frac{(1-\alpha)(3-5\alpha)}{12}.\]
Equality is attained in each case if and only if $f'(z)\,=\,(1-z)^{-2(1-\alpha)}$ or a rotation of this function.\\
\item[{\rm (b)}] If  $\alpha\in (1/5,1)$, then
\[|\Gamma_2(F)|\,\leq\,\frac{1-\alpha}{6}.\]
Equality is attained in each case if and only if $f'(z)\,=\,(1-z^2)^{-(1-\alpha)}$ or a rotation of this function.
\end{enumerate}
\end{theorem}
\begin{proof}
Using the sharp inequality $|c_2|\,\leq\,1\,-\,|c_1|^2$, we see that the above expression for $\Gamma_2(F)$ implies
\[4|\Gamma_2(F)|\,\leq\,\frac{1-\alpha}{3}\,(2\,+\,|c_1|^2(|3-5\alpha|-2))\]
In the first case the maximum of this expression is attained for $|c_1|\,=\,1$ and in the second case for $|c_1|\,=\,0$.
The extremal functions are calculated using $\varphi(z)\,=\,z$ in the first case and $\varphi(z)\,=\,z^2$ in the second case.

Concerning these inequalities compare \cite{LP14}.
\end{proof}


 In their paper \cite{PS81}, Prokhorov and  Szynal calculated the maximum of the expression
\[|c_3+\mu c_1c_2+\upsilon c_1^3|\]
for fixed $(\mu,\upsilon)\in \mathbb{R}^2$, where $\varphi$ varies in the set of Schwarz functions. It is obvious
that this result can be used to get the maximum of $|\Gamma_3(F)|$ for any $\alpha \in [-1/2,1)$. Since some of
these inequalities and their extremal functions are very much involved, we want to mention only those cases,
 where these expressions are nice. Hence, we only mention the related cases of the lemma of  Prokhorov and  Szynal.

\begin{lemma}\cite[Lemma~2]{PS81}\label{lem2-thm-C}
Let $\varphi(z)= \sum_{k=1}^{\infty}c_kz^k \in \mathcal B$ be a Schwarz function and
$$\Psi(\varphi)=|c_3+\mu c_1c_2+\upsilon c_1^3|.
$$
Then  we have the following sharp estimates:
\begin{enumerate}
\item[{\rm (a)}]
$\Psi(\varphi)\le  1 \mbox{ if }(\mu,\upsilon)\in D_1\cup D_2,$
where
\begin{align*}
  D_1 &= \left\{(\mu,\upsilon)\in\mathbb{R}^2:\, |\mu|\le \frac{1}{2},~~  -1\le \upsilon \,\le 1
 \right\}, ~\mbox{ and }\\[1mm]
  D_2 &= \left\{(\mu,\upsilon)\in\mathbb{R}^2:\,\frac{1}{2} \le |\mu|\le 2,~~
  {\frac{4}{27}(|\mu|+1)^3-(|\mu|+1)} \le \,\upsilon \,\le 1  \right\}.
 \end{align*}

\item[{\rm (b)}]
$\Psi(\varphi)\le |\upsilon| \mbox{ if }(\mu,\upsilon)\in D_6\cup D_7,$
where
\begin{align*}
  D_6 &= \left\{(\mu,\upsilon)\in\mathbb{R}^2:\, 2\le |\mu|\le 4,~~  \upsilon \,\ge
  \frac{1}{12}\,(\mu^2+8)\right\}, ~\mbox{ and }\\[1mm]
  D_7 &= \left\{(\mu,\upsilon)\in\mathbb{R}^2:\, |\mu|\ge 4,~~  \upsilon \,\ge
  \frac{2}{3}\,(|\mu|-1)  \right\}.
 \end{align*}
 \end{enumerate}
\end{lemma}
It is a lengthy, but straightforward verification that
\begin{equation}\label{d1234}
\left \{ \begin{array}{l}
(\mu,\upsilon) \in D_1, \mbox{\rm{if}}\,\,\, \alpha\in \left[\frac{1}{2},\frac{7}{10}\right],\\[2mm]
(\mu,\upsilon) \in D_2, \mbox{\rm{if}}\,\,\,\alpha\in {\left[0.21605468,\,\frac{1}{2}\right]},\\[2mm]
(\mu,\upsilon) \in D_6, \mbox{\rm{if}}\,\,\, \alpha\in \left[-\frac{1}{5},\frac{7}{47}\right],\\[2mm]
(\mu,\upsilon) \in D_7, \mbox{\rm{if}}\,\,\,\alpha\in \left[-\frac{1}{2},-\frac{1}{5}\right]
\end{array}
\right .
\end{equation}
which help to prove the next result.

\begin{theorem}\label{thm3-thm-C}
Let $f\in \F(\alpha)$ for $\alpha \in [-\frac{1}{2},1)$. Then
\[|\Gamma_3(F)|\,\leq\,\frac{1-\alpha}{12}, \quad \alpha\in {\left[0.21605468,\,\frac{7}{10}\right]}.\]
Equality is attained if $f'(z)\,=\,(1-z^3)^{-\frac{2(1-\alpha)}{3}}$ or a rotation of this function. Also,
\[|\Gamma_3(F)|\,\leq\,\frac{(1-\alpha)(3\alpha-2)(2\alpha-1)}{12}, \quad \alpha\in
{\left[-\frac{1}{2},\frac{7}{47}\right]}.\]
Equality is attained if $f'(z)\,=\,(1-z)^{-2(1-\alpha)}$ or a rotation of this function.
\end{theorem}
\begin{proof}
We see that from the expression (\ref{basic}) for $\Gamma_3(F)$ implies
\begin{align*}\label{eq3-thm-C}
\nonumber 6\,\left|\Gamma_3(F)\right|\,& = \frac{1-\alpha}{2}\, \Big| c_3 - (3-5\alpha)\,c_1\,c_2 + (3\alpha -2)(2\alpha -1)\,c_1^3 \Big | \\
 & =: \frac{1-\alpha}{2}\, |I_1|,
\end{align*}
where
\[  I_1= c_3+\mu\,c_1\,c_2 + \upsilon \, c_1^3,\quad \mu= 5\alpha-3 ~\mbox{ and }~
 \upsilon = (3\alpha -2)(2\alpha -1).
 \]
Our aim is to get a sharp bound for $|I_1|$. Lemma \ref{lem2-thm-C}(a) and \eqref{d1234} give $|I_1|\leq 1$ for
$D_1\cup D_2$ and the desired inequality follows.

Using the second part of Lemma \ref{lem2-thm-C} and \eqref{d1234}, we find that
$$|I_1|\le |\upsilon|=(3\alpha -2)(2\alpha -1)\, \mbox{ for } D_6\cup D_7.
$$
 This completes the proof of Theorem~\ref{thm3-thm-C}.
\end{proof}

If we take $\alpha=0$ and $\alpha=-1/2$ in Theorems~\ref{thm1-thm-C},\,\ref{thm2-thm-C} and \ref{thm3-thm-C}, then we get the
following interesting cases.

\begin{corollary}\label{cor3-thm-C}
Let $f\in \CC$. Then 
\[ |\Gamma_n(F)|\,\le \, \frac{1}{2n} \mbox{ for } n=1,2,3.
\]
The estimates are sharp for the function $l(z)=z/(1-z)$.
\end{corollary}

\begin{corollary}\label{cor4}
If $f\in \F(-1/2)$, then we have the sharp inequalities
\[ |\Gamma_1(F)|\,\le \, \frac{3}{4},\,
   |\Gamma_2(F)|\,\le \, \frac{11}{16},\,  \mbox{ and } \,
   |\Gamma_3(F)|\,\le \, \frac{7}{8}
\]
The estimates are sharp for the function $\displaystyle f_0(z)=\frac{z-z^2/2}{(1-z)^2}$.
\end{corollary}



\section{Concluding Remarks}

From Theorem \ref{PSW-thU}, we see that logarithmic inverse coefficients for the family
$\mathcal{U}(\lambda)$ for the remaining coefficients $\Gamma_n$ for $n\geq 3$ are open.

We recognized that in the case of convex functions $f\in \CC$,
\[ |\Gamma_n(F)|\,\le \, \frac{1}{2n}
\]
cannot be valid for  $n\geq 10$, although this is true for $n=1,2,3$ by Corollary~\ref{cor3-thm-C}. In fact, if this were true for $n\geq 10$,
then the third Lebedev-Milin inequality (see the book   by S. Gong \cite[p.~80]{Gong89}) would imply that the moduli of the coefficients of the inverses of convex functions are all less than $1$. But this is clearly wrong at least for $n \geq 10$ (see Kirwan and Schober
\cite{KirSch79}). On the other hand, it is natural to ask whether the last inequality is true for other values of $n$, namely, for $4\leq n\leq 9$. Finally, Corollary \ref{cor4} shows that analog problem for the class $\F(-1/2)$ is also open for $n\ge 4.$

\subsection*{Acknowledgements.}
This work was completed while the second author was at IIT Madras for a short period during July-August, 2018.
The work of the first author is supported by Mathematical Research Impact Centric Support of  Department of Science and Technology (DST),
India~(MTR/2017/000367). The second author thanks Science and Engineering Research Board, DST, India,
for its support by  SERB National Post-Doctoral Fellowship
(Grant No. PDF/2016/001274).

\end{document}